\documentclass{amsart}
\usepackage{xy, enumerate, pstcol, amssymb, amsfonts, amsbsy, amsthm, amsmath, amscd, epsfig, latexsym, mathtools, stmaryrd, epic, eepic, eucal, multicol, fancyhdr, graphicx}

\DeclareMathAlphabet{\mathpzc}{OT1}{pzc}{m}{it}

\newenvironment{dem}{\begin{proof}[\bf Proof]}{\end{proof}}
\newtheorem{theorem}{\bf Theorem}[section]
\newtheorem*{maintheorem}{\bf Main Theorem}
\newtheorem*{conjrecall}{\bf Conjecture \ref{conj:gens}}
\newtheorem{lemma}[theorem]{\bf Lemma}
\newtheorem{propos}[theorem]{\bf Proposition}
\newtheorem{corol}[theorem]{\bf Corollary}
\newtheorem{claim}[theorem]{\bf Claim}
\newtheorem{conj}[theorem]{\bf Conjecture}

\theoremstyle{definition}
\newtheorem{defi}[theorem]{\bf Definition}
\newtheorem{rmk}[theorem]{\bf Remark}

\newtheorem{notconv}[theorem]{\bf Convention}

\newcommand{\Dcal}{\mathcal D}

\newcommand{\G}{\mathbb G}

\newcommand{\Hcal}{\mathcal H}

\newcommand{\Mcal}{\mathcal M}

\newcommand{\Ocal}{\mathcal O}
\newcommand{\Pro}{\mathbb P}

\newcommand{\Q}{\mathbb Q}
\newcommand{\Qcal}{\mathcal Q}

\newcommand{\Z}{\mathbb Z}

\newcommand{\cart}{\ar @{} [dr] |{\Box}}

\newcommand{\GL}{\text{GL}}
\newcommand{\PGL}{\text{PGL}}

\newcommand{\Sym}{\text{Sym}}

\newcommand{\ov}{\overline}
\newcommand{\im}{\text{Im}}

\newcommand{\wt}{\widetilde}

\newcommand{\Sch}{\text{Sch}}

\input xy
\xyoption{all}

\title{The integral Chow ring of $\Mcal_{0}(\Pro^r, d)$, for $d$ odd}
\author{Renzo Cavalieri}
\author{Damiano Fulghesu}
\date{\today}

\address{Renzo Cavalieri, Department of Mathematics, Colorado State University, 1874 Campus Delivery, Fort Collins, CO, 80523-1874, U.S.A.}
\email{renzo@math.colostate.edu}
\address{Damiano Fulghesu, Department of Mathematics, Minnesota State University, 1104 7th Ave South, Moorhead, MN 56563, U.S.A.}\email{fulghesu@mnstate.edu}

\begin{document}

\begin{abstract}
For any odd integer $d$, we give a presentation for the integral Chow ring of the stack $\Mcal_{0}(\Pro^r, d)$, as a quotient of the polynomial ring $\Z[c_1,c_2]$. We describe an efficient set of generators for the ideal of relations, and compute them in generating series form.  
The paper concludes with explicit computations of some examples for low values of $d$ and $r$, and a conjecture for a minimal set of generators.
\end{abstract}

\maketitle

\section{Introduction}
\subsection{Results and strategy of proof}
The main goal of this work is to compute the integral Chow ring of the open set of the space of rational stable maps of odd degree to projective space, consisting of maps from  irreducible source curves. We exhibit such ring as the quotient of a polynomial ring in two variables, and provide an efficient set of generators for the ideal of relations. We state here the complete result, which is proven in several steps throughout the paper.

\begin{maintheorem}[Thm.\ref{thmheresthepres}, Thm. \ref{ideal.of.relations}, Thm. \ref{thm:genfunctionsforfirstenveloperelations}, Thm. \ref{thm:degszero}]
For $r$ a positive integer, and $d$ an odd positive integer,
a presentation for the integral Chow ring of the stack $\Mcal_{0}(\Pro^r, d)$ is given by:
\begin{equation*}
 A^\ast(\Mcal_{0}(\Pro^r, d)) = \frac{\Z[c_1,c_2]}{(\alpha^{r,d}_{i,k})_{i,k}},
 \end{equation*}
 where $c_i$ is a graded variable of degree $i$ and $\alpha_{i,k}^{r,d}$ is homogeneous of degree $ir+k$ (Thm. \ref{thmheresthepres}); in order to generate the ideal of relations it is sufficient to consider the following values for $i,k$ (Thm. \ref{ideal.of.relations}):
 \begin{itemize}
     \item $i = 1$ and $k = 0,1$.
     \item $i$ is a prime power and $k = 0$.
 \end{itemize}
 
 \noindent The generating relations are exhibited in generating function form, where: $$\mathcal{A}_{i,k}(d) := \sum_{r=0}^{\infty} \alpha_{i,k}^{r,d}.$$
 
\noindent Relations of the form $\alpha^{r,d}_{1,k}$ are explicitly computed in Theorem \ref{thm:genfunctionsforfirstenveloperelations}:
 \begin{equation*}
  \mathcal{A}_{1,0}(d)  = \frac{d}{(1+\frac{d-1}{2}c_1)(1-\frac{d+1}{2}c_1)+d^2c_2} 
\end{equation*}
\begin{equation*}
 \mathcal{A}_{1,1}(d)  = \frac{1+\frac{d-1}{2}c_1}{(1+\frac{d-1}{2}c_1)(1-\frac{d+1}{2}c_1)+d^2c_2}-1.   
\end{equation*}

\noindent Relations of the form $\alpha^{r,d}_{i,0}$ are computed in Theorem \ref{thm:degszero} as:
\begin{equation*}
 \mathcal{A}_{i,0}(d)=
     \displaystyle \sum_{j=0}^{i} \frac{(-1)^{j}}{i! (l_1-l_2)^i}\binom{i}{j}\left[\left(1-\frac{\prod_{k = 0}^{d}\left(\frac{c_1}{2}+\left(k-\frac{d}{2}\right)(l_1-l_2) \right)}
     {\prod_{k = j}^{d-i+j}\left(\frac{c_1}{2}+\left(k-\frac{d}{2}\right)(l_1-l_2) \right)}
     \right)^{-1} -1\right], 
\end{equation*}
where $-l_1, -l_2$ are the Chern roots of the variables $c_1, c_2$, i.e. $c_1 = -l_1-l_2$ and $c_2 = l_1l_2$.
\end{maintheorem}

The integral Chow rings of the spaces $\Mcal_{0}(\Pro^r, d)$ exhibit a remarkable combinatorial structure. In Corollary \ref{cor:pol} it is shown that  for fixed $i,r$ the relations of the form $\alpha_{i,0}$ are polynomial in $d$, while in Corollary \ref{cor:diffop} one sees that, for fixed $d,r$, all relations of the form $\alpha_{i,0}$ may be extracted from a single two-variable monomial via the action of a differential operator, and  an appropriate Hadamard product of multi-variate power series. 

The proof of the Main Theorem is organized into three main stages: the first one consists of producing a presentation of $A^\ast(\Mcal_{0}(\Pro^r, d))$, where the ideal of relations has a very redundant generating set. The second consists in eliminating a large number of redundant generators, and in the third we compute the remaining relations. 

To obtain a presentation of $A^\ast(\Mcal_{0}(\Pro^r, d))$, we exhibit $\Mcal_{0}(\Pro^r, d)$ as a global quotient stack $[\widehat{U}_{r,d}/\GL_2]$ of an open set in affine space by the action of the group $\GL_2$, thus reducing the task to computing the $\GL_2$ equivariant Chow ring of $\widehat{U}_{r,d}$. It is at this moment that the hypothesis of $d$ being odd becomes necessary. By considering $\widehat{U}_{r,d}$ as a $\G_m$-bundle over its image in projective space, we reduce the computation of the equivariant Chow ring of $\widehat{U}_{r,d}$ to that of its image $U_{r,d}$: this is an open set in projective space, and its complement is covered by a $\GL_2$-equivariant envelope, whose connected components  $Z_i\cong \Pro^i\times \Pro^N$ are isomorphic to products of projective spaces. Using the standard excision sequence, we exhibit the Chow ring of $\Mcal_{0}(\Pro^r, d)$ as the quotient of the Chow ring of the ambient projective space by the ideal generated by the images of the pushfoward maps from the components of the envelope.

Elementary arguments suffice to show that the pushforward ${\pi_{i}}_\ast(h_i^k)$  of powers of the hyperplane classes from the left factor of the $Z_i$'s, with degree bounded by $i$, generate the ideal of relations.
In order to further reduce the set of necessary generators, a more subtle analysis is needed. One may replace the powers $h_i^k$  in the $i$-th component of the envelope with  suitably chosen monic polynomials in $h_i$. In many cases, such polynomials may be seen as arising from classes from lower envelopes, leading to an inductive argument that allows to narrow down the indispensable generators to those mentioned in the Main Theorem. The generating set we obtain is still not minimal, as can be seen from the computations in Section \ref{sec:planecubics}. Experimental computations led to the following conjecture.

\begin{conjrecall}
Consider the presentation of $ A^\ast(\Mcal_{0}(\Pro^r, d))$ from Theorem \ref{thmheresthepres}.  A minimal set of generators for the ideal of relations  is given by $\alpha_{1,0}^{r,d}, \alpha_{1,1}^{r,d}$ and the $\alpha_{p,0}^{r,d}$'s  where $p$ runs over all primes that divide $d$.
\end{conjrecall}

Regardless of the optimal generating set, we set forth to compute all relations coming from the first component of the envelope or from the pushforward of fundamental classes of any other component of the envelope.
 We use two different techniques to compute the two types of relations. For the relations coming from the first envelope,  the embedding of $\Pro^r$ as a coordinate hyperplane in $\Pro^{r+1}$ gives the relations a recursive  structure which allows to reconstruct them for all values of $r$ from the degenerate $r=0$ case. Encoding the relations in generating function form, the recursions become a linear system of functional equations which is easily solved. Such recursive structure is present for the relations coming from higher envelopes, but it becomes substantially more computationally intensive to use this technique to extract the relations. Hence, for the fundamental class relations, we instead use the Atyiah-Bott localization theorem on the left factor of the envelope $Z_i$ to obtain an expression for the fundamental class supported at the fixed points. Such an expression is then readily pushed forward to obtain the formulas in the Main Theorem\footnote{Using the Atyiah-Bott localization might in principle cause the loss of some torsion classes, but we are applying the theorem to projective space, whose integral Chow ring is known to not have torsion.}. The drawback of this technique is that the answer is not produced immediately as a polynomial in $c_1, c_2$, but rather as a (non obviously) symmetric polynomial in the Chern roots $l_1, l_2$. It would be interesting to symmetrize formula \eqref{eq:closedformzerorel}.

\subsection{History, motivation, and considerations}

One of the key conceptual leaps in modern enumerative geometry has been the translation of enumerative questions into  intersection theory on appropriate moduli spaces of geometric objects. Thus, for example, the unique conic through five points in the plane is obtained as the intersection of five hyperplanes in the $\Pro^5$ of all conics, and twelve rational cubics through eight points arise as the intersection of eight hyperplanes and the discriminant hypersurface in the  $\Pro^9$ of plane cubics.

The rich and rapid development of the field  that followed in the late 1800's led to many exciting computations, some  of which were alas  incorrect; these early mistakes brought awareness of many delicate issues and technical difficulties in implementing the intuitive plan of counting geometric objects by intersecting subvarieties in moduli spaces. On the one hand, this led to Hilbert's fifteenth problem \cite{Hilbert}, requesting rigorous foundations for Schubert calculus, which we now understand as intersection theory on Grassmannians and flag varieties. On the other, new perspectives on geometric objects and their moduli were developed to tackle even the most classical problems. For example, the classical enumerative problem of counting the number of plane rational curves of degree $d$ through $3d-1$ general points was solved in \cite{Kon:WC} by introducing the moduli spaces of stable maps, that shifted the perspective by thinking of plane curves as the images of functions from abstract curves.

The main object of study in algebraic intersection theory \cite{Ful, 3264} is the Chow ring, a codimension-graded ring generated by equivalence classes of closed subvarieties up to rational equivalence, where the product  extends the notion of transverse intersection of subvarieties. A full understanding of the Chow ring of a moduli space gives in principle access to any enumerative geometric problem involving the geometric objects described by the moduli space. It should come as no surprise then that Chow rings of moduli spaces are typically very sophisticated, and hard to compute objects. To further complicate things, automorphisms of geometric objects cause most moduli spaces to be represented only by stacks, rather than varieties or schemes. 

Working with rational coefficients gives the significant advantage that the Chow ring of a Deligne-Mumford stack agrees with that of its coarse moduli space, see \cite[Theorem 4.40]{Ed:eg}. Still, the Chow ring of  moduli spaces of curves $\mathcal{M}_g$ are known only up to $g = 9$ \cite{Mum, Fab:Ch1, Fab:Ch2, Iza, PV, CanLar}, and for $\overline{\mathcal{M}}_g$  up to $g = 3$ \cite{Mum, Fab:Ch1}.

While  rational coefficients are often sufficient for any application to enumerative problems, they cause to lose all torsion classes, arguably containing interesting information about the geometry of the moduli space. For example, the stack of smooth hyperelliptic curves of genus $g$ has coarse moduli space which is a finite quotient of $M_{0,2g+2}$ (an open subset of $\mathbb{A}^{2g-1})$, and has therefore trivial Chow ring, i.e. $A^\ast(\Hcal_g, \Q) = \Q$.

Edidin and Graham (\cite{EG}) generalize work of Totaro (\cite{Tot:CR}) and  approach the study of Chow rings of moduli spaces with integer coefficients via equivariant geometry: if a moduli space is presented as a global stack quotient $[X/G]$, then its integral Chow ring is the equivariant Chow ring $A^\ast_G(X).$ 

This perspective was used to unveil rich torsion structure in the integral Chow ring of hyperelliptic loci \cite{EFh, DiLh}, and in the locus of non-hyperelliptic curves of genus $3$ \cite{DLFV}. The only moduli spaces of curves for which the integral Chow ring has been computed  are $\overline{\Mcal}_{2}$ (\cite{LarM2bar}) and $\overline{\Mcal}_{2,1}$ (\cite{DilPVis}).

Besides hyperelliptic loci, the only infinite family of moduli spaces for which the integral Chow ring is known is  given by Grassmannians $G(k,n)$. A structurally satisfying presentation for all these Chow rings is given, for example, in \cite[Theorem 5.26]{3264}
: $A^{\ast}(G(k,r+1), \Z)$ is a quotient of the polynomial ring $\Z[c_1, \ldots, c_k]$, where $c_i$ is a variable of degree $i$. The ideal of relations is generated by the homogeneous terms of degrees between $r+2-k$ and  $ r+1$ in the power series
expansion of\footnote{For us $c_i$ denotes the $i$-th Chern class of the standard representation of $GL_n$; in \cite{3264} they are choosing the dual of the standard representation, hence the different signs in the formulas.} 
\begin{equation} \label{Gras:gf}
    \frac{1}{1-c_1+c_2-\ldots +(-1)^k c_k}.
\end{equation}
We view \eqref{Gras:gf} as a generating function for the ideal of relations of all Grassmannians of $k$ planes as the ambient dimension $r+1$ varies.

When $k=2$, we have $Gr(2,r+1) \cong \Mcal_{0}(\Pro^r, 1)$, and it is easily seen that the terms of degree $r$ and $r+1$  of  $1/(1-c_1+c_2)$ agree with the generating functions $\mathcal{A}_{1,0}(1)$ and $\mathcal{A}_{1,1}(1)$ from the Main Theorem. 

In \cite{Pandhanonlin}, Pandharipande considers the function $\Phi_f: Gr(2,r+1) = \mathcal{M}_{0}(\Pro^r, 1)\to \mathcal{M}_{0}(\Pro^r, d)$ induced by post-composition with a fixed degree $d$ map $f:\Pro^r\to \Pro^r$. He shows that when working with rational coefficients, the pull-back 
\begin{equation}
   \Phi_f^\ast: A^\ast(\Mcal_{0}(\Pro^r, d),\Q)\to  A^\ast(\Mcal_{0}(\Pro^r, 1),\Q)
\end{equation}
is an isomorphism. It is interesting to observe how this result relates to the Main Theorem. First off, all relations $\alpha_{i,0}$ with $i>1$ are $i$-torsion, and therefore vanish when tensoring  coefficients with $\Q$. In order to check that relations from the first envelope agree, one must analyze the lift 
$$
\widehat{\Phi}_f: \widehat U_{r,1}\to \widehat U_{r,d},
$$
where the $\GL_2$-action on the two spaces is as in Proposition \ref{prop.reduction.GL2}. A non-trivial endomorphism $\varphi_f$ of $\GL_2$ is required to make the map $\widehat{\Phi}_f$ equivariant, and this induces the transformation $\varphi_f^\ast(c_1) = \frac{c_1}{d}, \varphi_f^\ast(c_2) = c_2 -\frac{(d^2-1)c_1^2}{4d^2}$. It is then immediate to check that the relations agree up to global factors which are powers of $d$ (irrelevant after tensoring with $\Q)$.

We make two important assumptions on the base field $k$. The first one is that $k$ is algebraically closed. We use this condition to unambiguously determine the degree of certain maps in Lemma \ref{lemma:coefficient-alpha-tilde}. Moreover, we assume that the base field has characteristic greater than $d$. This condition allows to work with localization formulas and to construct formula (\ref{eq:genfunzerorel}).

\subsection{Notation}
There are four main discrete invariants that  take different values in this work: 
\begin{description}
\item[$r$] the dimension of the target projective space;
\item[$d$]  the degree of the  map $f:\Pro^1\to \Pro^r$ or equivalently of the $(r+1)$ defining polynomials;
\item[$i$] the component of the envelope parameterizing polynomials that contain a common factor of degree $i$;
\item[$k$] the power of the hyperplane class on the left factor of $\Pro(W_i)\times \Pro(W_{d-i}^{\oplus r+1})$.
\end{description}
As a  consequence, the generators of the ideal of relations for the presentation of  $A^\ast(\Mcal_{0}(\Pro^r, d))$ are polynomials $\alpha_{i,k}^{r,d}$ depending on four indices. 
\begin{notconv}To lighten this notation, we adopt the convention of suppressing indices corresponding to discrete invariants that remain fixed throughout a section. Thus for example in Section \ref{sec:genrel}, where $r$ and $d$ are fixed, the relations are  denoted $\alpha_{i,k}$; whereas in Sections \ref{sec:relfromz1}, where we consider maps comparing different $\Pro^r$'s, we mantain the superscript $r$. A similar convention is adopted for any other quantity depending on these four discrete invariants. 
\end{notconv}

The Chow ring  $A^\ast(\Mcal_{0}(\Pro^r, d))$ is presented as a quotient of the equivariant Chow ring of $\Pro(W_{d}^{\oplus r+1})$, where one of the relations is $H - (d+1)c_1/2$. One may then eliminate the hyperplane class $H$ and regard $A^\ast(\Mcal_{0}(\Pro^r, d))$ as a quotient of the polynomial ring $\Z[c_1, c_2]$. We adopt the following notation to deal with this phenomenon.
\begin{notconv}\label{convsubstitute}
We denote the class $\pi_{i_\ast}(h_i^k)$ by $\alpha_{i,k}(H)$ when we wish to regard it as a Chow class on $\Pro(W_{d}^{\oplus r+1})$, by $\alpha_{i,k}$  when we wish to regard it as a Chow class on $\Mcal_{0}(\Pro^r, d)$, so that $\alpha_{i, k}=\alpha_{i,k}((d+1)c_1/2)$.
This convention extends to any other class with such double personality.
\end{notconv}

Some of the commonly recurring notation   throughout the paper is collected in Table \ref{tab:notat}.

\begin{table}[]
\centering
\begin{tabular}{l|l}
$W_d$ & homogeneous polynomials of degree $d$ in two variables\\
$U_{r,d}\subset \Pro(W_d^{r+1})$ & $(r+1)$-tuples of polynomials containing no common factor of positive degree\\
$\Delta_{r,d}\subset \Pro(W_d^{r+1})$ & the complement of $U_{r,d}$\\
$\widehat{U}_{r,d}$ & the affine cone over $U_{r,d}$\\
$l_1, l_2$ & torus weights, Chern roots for the dual of the standard representation of $\GL_2$\\
$c_1,c_2$ & Chern classes of the standard representation of $\GL_2$\\
$H$ & equivariant hyperplane class for $\Pro(W_d^{r+1})$\\
$P_{r,d}(H)$ & basic relation in equivariant Chow ring of $\Pro(W_d^{r+1})$, see \eqref{eq:Prd}\\
$P_{d}$ & $P_{0,d}((d+1)c_1/2)$\\
$Z_i\subset \Delta_{r,d}$ & stratum of $(r+1)$ polynomials with largest common factor of degree $i$\\
$\wt{Z}_i$ & $\Pro(W_i) \times \Pro \left( W_{d-i}^{\oplus r + 1} \right)$, with a natural map $\pi_i$, an isomorphism onto the closure of $Z_i$\\
$h_i\in A^*_T(\wt{Z}_i)$ & equivariant hyperplane class for left factor\\
$\eta_i\in A^*_T(\wt{Z}_i)$ & equivariant hyperplane class for right factor\\
$\alpha_{i,k}^{r,d}$ & $\pi_{i_\ast}(h_i^k)$
\end{tabular}
    \caption{Recurring notation.}
     \label{tab:notat}
\end{table}

\subsection{Acknowledgements} We would like to thank Seth Ireland for writing the code that gave extensive verifications of Conjecture \ref{conj:gens}, and Chris Peterson for several conversations related to the project. We are grateful to Rahul Pandharipande, Jim Bryan, Tom Graber and the participants of the {\it intercontinental algebraic geometry and moduli zoominar} as well as Ravi Vakil, Jesse Kass, Peter Petrov and the participants of the {\it Stanford virtual algebraic geometry seminar} for the many interesting questions and observations following the presentation of this work. We gratefully acknowledge the following funding institutions: R.C. is supported by NSF grant DMS-2100962,  D.F. by Simons collaboration grant 360311.

\section{A Presentation of the Stack $\Mcal_{0}(\Pro^r, 2s+1)$}

Let $k$ be an algebraically closed field of characteristic 0 or larger than $2s+1$. We refer the reader to \cite{FP} for a general definition of the stack of stable maps $\Mcal_{g,n}(X,\beta)$ where $X$ is a projective variety and $\beta$ is a class in $A_1(X,\Z)$. In this paper we focus on the special case $\Mcal_{0}(\Pro^r,d)$ of maps from rational curves to projective space. For completeness, we recall the  definition.
\begin{defi}
Let $\Mcal_{0}(\Pro^r,d)$ be the category whose objects are diagrams
\[
\xymatrix{
C \ar[d]_{\pi}  \ar[r]^{f} & \Pro^r\\
S
}
\]
such that:
\begin{itemize}
\item $S$ is a $k$-scheme,
\item the morphism $\pi: C \to S$ is a projective flat family of curves isomorphic to $\Pro^1$,
\item the degree of $f^*\Ocal_{\Pro^r}(1)$ on the geometric fibers of $\pi: C \to S$ is $d$.
\end{itemize}
As a consequence of the third bullet point,  the restriction of $f$ on the geometric fibers of $\pi: C \to S$ is in particular not constant. Arrows are cartesian diagrams:
\[
\xymatrix{
C_1 \cart \ar[r]^{\psi} \ar[d]_{\pi_1} \ar@/^2.0pc/[rr]^{f_1=f_2 \circ \psi}& C_2 \ar[d]^{\pi_2}  \ar[r]^{f_2} & \Pro^r\\
S_1 \ar[r]^{\varphi}& S_2
}
\]
\end{defi}

The category $\Mcal_{0}(\Pro^r,d)$ is naturally a category fibered in groupoids over the category $(\Sch/k)$ of $k$-schemes. It is a proper Deligne-Mumford stack, see, for example, \cite[Theorem 3.14]{BM}.
As an example of a moduli point with non-trivial isotropy, let $f_d: \Pro^1 \to \Pro^1$ be defined as $f_d(x:y)  = (x^d:y^d)$, and let $\iota_L:\Pro^1\to \Pro^r$ be the inclusion of $\Pro^1$ as a line in $\Pro^r$; then $(\Pro^1, \iota_L\circ f_d)\in \Mcal_{0}(\Pro^r,d)$  has isotropy group $\mu_d$.

Observe that $\Mcal_{0}(\Pro^r,d)$ has a structure of a global quotient by a linear algebraic group. Let $E$ be the standard representation of $\GL_2$. We identify the projective line as
\[
\Pro^1 \cong \Pro(E).
\]
Consider the vector space of homogeneous forms of degree $d$ over $\Pro^1$
\[
W_d:= H^0(\Pro^1,\Ocal_{\Pro^1}(d)) = \Sym^d(E^\vee)
\]
where $E^\vee$ is the dual of the representation $E$.

The set of regular maps 
$
\Pro^1 \to \Pro^r
$
of degree $d$ is an open subset $U_{r,d}$ of $\Pro \left( W_d^{\oplus r+1} \right)$.  More precisely, a general element $f$ of $U_{r,d}$ can be written as
\[
f= \left [ f_1(x,y), \dots, f_{r+1} (x,y)  \right]
\]
where
\begin{itemize}
\item for all $i=1, \dots, r+1$, the polynomial in two variables $f_i(x,y)$ is homogeneous of degree $d$;
\item the map $f$ is free of base points, in other words, the polynomials $f_i(x,y)$ have no common factors.\end{itemize}

Call $\Delta_{r,d}$ the complement of $U_{r,d}$ in $\Pro \left( W_d^{\oplus r+1} \right)$, and $\widehat{U}_{r,d}$ the affine cone over $U_{r,d}$, that is the preimage of $U_{r,d}$ via the tautological map
\[ 
W_d^{\oplus r+1} \backslash 0 \to \Pro \left( W_d^{\oplus r+1} \right).
\]

One has the following isomorphism of stacks (see \cite[Section 2]{FP} for a more general result):

\begin{equation}\label{basic.isom}
\Mcal_{0}(\Pro^r, d) \cong \left[ U_{r,d} / \PGL_2 \right]
\end{equation}\label{eq:actpgl2}
where the action of $\PGL_2$ on $U_{r,d}$ is given by:
\begin{equation}
(A \cdot [f_1, f_2, \dots, f_{r+1}])(x,y) = [f_1 (A^{-1}(x,y)), f_2(A^{-1}(x,y)), \dots, f_{r+1}(A^{-1}(x,y))].
\end{equation}
The use of the inverse matrix is necessary to have a well defined left action. 
\begin{lemma}\label{lemma.normal.subgroup}
Let $H$ be a normal subgroup of a linear group $G$ and let $X$ be a quasi-projective scheme equipped with a $G$-action. Assume that $H$ acts freely on $X$ so that the quotient $X/H$ is a quasi-projective scheme as well. Then we have an isomorphism of quotient stacks:
\[
\left[ X / G \right] \cong \left[ (X/H) / (G/H) \right]
\]
\end{lemma}

\begin{dem}
See, for example, \cite[Example 3.3]{Hei05}.
\end{dem}

\begin{rmk} A consequence of Lemma \ref{lemma.normal.subgroup} is the induced isomorphism of equivariant intersection rings:
\[
A^*_G(X) \cong A^*_{G/H}(X/H).
\]
For a direct proof of this isomorphism see \cite[Lemma 2.1]{MV}.
\end{rmk}

From now on, assume the degree $d$ to be odd.
\begin{propos}\label{prop.reduction.GL2}
Let $d=2s+1$ be a positive odd integer. The stack $\Mcal_{0}(\Pro^r, d)$ is isomorphic to the quotient stack
\[
\Mcal_{0}(\Pro^r, d) \cong \left[ \widehat{U}_{r,d} / \GL_2 \right]
\]
where the action of $\GL_2$ is
\[
A \cdot (f_0, f_1, \dots, f_{r}) (x,y) = \det{A}^{s+1} (f_0 (A^{-1}(x,y)), f_1(A^{-1}(x,y)), \dots, f_{r}(A^{-1}(x,y))).
\]
\end{propos}

\begin{dem}
Apply Lemma \ref{lemma.normal.subgroup} with the substitutions:
\[
X= \widehat{U}_{r,d}; \quad G=\GL_2; \quad H=\G_m.
\]
 $\G_m$ is normal in $\GL_2$; one must describe the induced action of $\G_m$ on $\widehat{U}_{r,d} \subset W_d^{\oplus r+1}$.

Refer to the exact sequence of groups:
\[
1 \to \G_m \xrightarrow{\varphi} \GL_2 \xrightarrow{\pi} \PGL_2 \to 1
\]

Let $\lambda \in \G_m$ and $\varphi(\lambda)= \left [ \begin{array}{cc} \lambda & 0\\ 0 & \lambda \end{array} \right ]=:A$, we have
\begin{eqnarray*}
A \cdot (f_1, f_2, \dots, f_{r+1}) (x,y) &=& \det{A}^{s+1} (f_1 (A^{-1}(x,y)), f_2(A^{-1}(x,y)), \dots, f_{r+1}(A^{-1}(x,y))) =\\
& = & \lambda^{2(s+1)} \frac{1}{\lambda^{d}}(f_1, f_2, \dots, f_{r+1}) (x,y) = \lambda (f_1, f_2, \dots, f_{r+1}) (x,y),
\end{eqnarray*}
therefore $\widehat{U}_{r,d}/\G_m = U_{r,d}$. Notice that the action of $\G_m$ is free on $\widehat{U}_{r,d}$.  Lemma \ref{lemma.normal.subgroup} implies
\[
\left[ \widehat{U}_{r,d} / \GL_2 \right] \cong \left[ U_{r,d} / \PGL_2 \right],
\]
where the action of $\PGL_2$ over $U_{r,d}$ is  as in \eqref{eq:actpgl2}. We conclude thanks to isomorphism (\ref{basic.isom}).
\end{dem}

\section{Generators and  Relations}
\label{sec:genrel}

By Proposition \ref{prop.reduction.GL2}, we are reduced to computing the equivariant intersection ring $A^*_{\GL_2}\left( \widehat{U}_{r,d} \right)$.  
First, we recall some facts about equivariant Chow rings we will use, and set notation. Let $T$ denote the maximal torus for $\GL_2$ represented by diagonal matrices, and consider the induced morphism $Bi: BT \to B\GL_2$.  We denote by $E$  the standard representation of $\GL_2$, which we think of as (the pull-back to the point via the quotient map $pt \to B\GL_2$ of) a rank two vector bundle over $B\GL2$. Since  the Chern classes of $E$ will be frequently used, we denote $c_i(E)$ simply by $c_i$. The pull-back $Bi^\ast(E^\vee)$ splits as the direct sum of two line bundles on $BT$: the characters  $\lambda_1, \lambda_2$ are given by the two coordinate projections of $T$.

Denoting $A^{*}_T = A^{ *}_T(pt.)$, it is known (see e.g.  \cite[Chapter 4]{Mirror-symmetry-book}) that $A^{*}_T=\Z\left[  l_1, l_2  \right]$, with $l_i =  c_1\left( \lambda_i \right)$. By a slight abuse of notation we also denote by $l_i $  the Chern roots of the vector bundle $E^\vee$, since we have
\begin{eqnarray*}
Bi^\ast c_1&=&-(l_1+l_2)\\
Bi^\ast  c_2 &=& l_1 l_2.
\end{eqnarray*}
 The Weyl group $S_2$  acts on $A^*_T$ by permuting the classes $l_i$ and  $A^*_{\GL_2}=\left( A^*_T \right)^{S_2}$ (see \cite[Proposition 6]{EG}).

Consider the following commutative $\GL_2$-equivariant diagram
\[
\xymatrix{
\widehat{U}_{r,d} \ar[r]^(0.40){j} \ar[d]^{\pi} & W_d^{\oplus r+1} \ar[d]^{\pi} \backslash 0\\
U_{r,d} \ar[r]^(0.40){j} & \Pro \left( W_d^{\oplus r+1} \right)
}
\]
where the horizontal arrows are the natural open inclusions and the vertical arrows are the natural quotient maps by the action of $\G_m$. The vertical maps may be interpreted as principal $\G_m$-bundles associated to the $\GL_2$-equivariant line bundle $\Dcal^{\otimes s+1} \otimes \Ocal(-1)$, where $\Ocal(-1)$ is the tautological bundle over $\Pro \left( W_d^{\oplus r+1} \right)$ and $\Dcal$ is the one dimensional representation of $\GL_2$ associated to the determinant $\bigwedge^2 E$. The first Chern class of $\Dcal$ is $c_1(\Dcal)=c_1$, therefore we have
\[
c_1 \left( \Dcal^{\otimes s+1} \otimes \Ocal(-1) \right)=(s+1)c_1 - H
\]
where $H$ is the canonical equivariant lift of the hyperplane class of $\Pro \left( W_d^{\oplus r+1} \right)$. By arguing as in Lemma 3.2 of \cite{EFh}, the pull-back morphism
\[
\pi^*: A^*_{\GL_2} \left( U_{r,d} \right) \to A^*_{\GL_2} \left( \widehat{U}_{r,d} \right)
\]
is surjective and its kernel is generated by $H-(s+1)c_1$. We may therefore determine  $A^*_{\GL_2} \left( \widehat{U}_{r,d} \right)$ from a presentation of the ring $A^*_{\GL_2} \left( U_{r,d} \right)$ by applying the substitution $H=(s+1)c_1$.

In order to compute $A^*_{\GL_2} \left( U_{r,d} \right)$, we consider the following exact sequence of $A^*_{\GL_2}$-modules:
\[
A^*_{\GL_2}\left( \Delta_{r,d} \right) \xrightarrow{i_*} A^*_{\GL_2}\left( \Pro \left( W_d^{\oplus r+1} \right) \right) \xrightarrow{j^*} A^*_{\GL_2}\left( U_{r,d} \right) \to 0.
\]

Using standard techniques as in \cite[Section 3.2]{Fu-Vi}, we have the following isomorphism:
\begin{equation}
A^*_{\GL_2}\left( \Pro \left( W_d^{\oplus r+1} \right) \right) \cong \frac{\Z[c_1,c_2,H]}{(P_{r,d}(H))}
\end{equation}
where
\begin{equation}\label{eq:Prd}
P_{r,d}(H)= \prod_{k=0}^d (H+(d-k)l_1 + kl_2)^{r+1}.
\end{equation}

In conclusion, we have:
\begin{equation}\label{isom}
A^* \left( \Mcal_{0}(\Pro^r, d) \right)
\cong \frac{A^*_{\GL_2}\left( \Pro \left( W_d^{\oplus r+1} \right) \right)}{(\im(i_*), H-(s+1)c_1)} \cong \frac{\Z[c_1,c_2,H]}{(\im(i_*), H-(s+1)c_1,P_{r,d}((s+1)c_1))},
\end{equation}
where $\im(i_*)$ is the image of the push-forward $i_*$.
The goal is now to compute generators for the ideal $\im(i_*)$.

Let $X$ be a $G$-scheme and $\Delta \xrightarrow{i} X$ an equivariant closed embedding. A consolidated method to determine the image of the group homomorphism
\[
A_G^*(\Delta) \xrightarrow{i_*} A_G^*(X),
\]
is to use a so-called equivariant envelope of $\Delta$. Recall that an {\it envelope} $\wt{\Delta} \xrightarrow{\widetilde{\pi}} \Delta$, see \cite[Definition 18.3]{Ful}, is a proper map such that for every subvariety $V$ of $\Delta$, there is a subvariety $\wt{V}$ of $\wt{\Delta}$ such that the morphism $\widetilde{\pi}$ maps $\wt{\Delta}$ birationally onto $V$. In the category of $G$-schemes, $\widetilde{\pi}$ is called an {\it equivariant envelope}, see \cite[Section 2.6]{EG}, if $\widetilde{\pi}$ is $G$-equivariant and if one can choose $\wt{V}$ to be $G$-invariant whenever $V$ is $G$-invariant. An equivariant envelope $\wt{\Delta} \xrightarrow{\widetilde{\pi}} \Delta$ for a closed subscheme $\Delta \xrightarrow{i} X$ is especially helpful when one can explicitly describe the Chow group of $\wt{\Delta}$ and the image of the group homomorphism
\[
A_G^*(\wt{\Delta}) \xrightarrow{(i \circ \widetilde{\pi})_*} A_G^*(X).
\]
Since the group homomorphism $\widetilde{\pi}_\ast$ is surjective (see \cite[Lemma 3]{EG} and \cite[Lemma 18.3(6)]{Ful}), one has the following result.

\begin{theorem}\label{thm:ipi}
Let $X$ be a $G$-scheme and $\Delta \xrightarrow{i} X$ an equivariant closed embedding. If $\wt{\Delta} \xrightarrow{\widetilde{\pi}} \Delta$ is an equivariant envelope of $\Delta$, then 
\[
(i \circ \widetilde{\pi})_* \left( A_{G}^* ( \wt{\Delta}) \right) = { i_* }\left( A_{G}^* \left( \Delta \right) \right).
\]
\end{theorem}

Returning to the computation, we  construct an equivariant envelope for the locus of degenerate maps $\Delta_{r,d}$.
\begin{defi}
For every $i=1, \dots, d$, denote by  $Z_i$  the subspace of $\Pro \left( W_d^{\oplus r+1} \right)$ representing $(r+1)$-tuples of polynomials with a common factor of degree $i$, but not $i+1$:
\begin{equation}
 Z_i = \{(f_1, \ldots, f_{r+1}) | \deg(\gcd(f_1, \ldots, f_{r+1}) = i \}.
\end{equation}
\end{defi}

The family $\{ Z_i\}_{i=1, \dots, d}$ is an equivariant stratification of $\Delta_{r,d}$, in the sense of  \cite[Definition 1.2]{DLFV}.

Define
\begin{equation}
 \wt{Z}_i := \Pro(W_i) \times \Pro \left( W_{d-i}^{\oplus r + 1} \right)   
\end{equation}
and the morphisms
\begin{eqnarray}\label{eq:defincl}
\pi_i: \wt{Z}_i  &\to & \Pro \left( W_d^{\oplus r+1} \right) \nonumber \\
([c(x,y)],[g_1(x,y), g_2(x,y), \dots, g_{r+1}(x,y)]) &\mapsto & [c(x,y) \cdot g_1(x,y), c(x,y) \cdot g_2(x,y), \dots, c(x,y) \cdot g_{r+1}(x,y)].
\end{eqnarray}

\begin{propos}\label{prop.equivariant.envelope}

Denote by $\wt{Z} := \bigsqcup_{i=1, \dots, d} \wt{Z}_i$ and by $\pi := \sqcup \pi_i$. 
Denoting by $\wt\pi$ the restriction of $\pi$ onto its image $\Delta_{r,d}$, we have that $$\wt\pi: \wt{Z}\to \Delta_{r,d}$$ is an equivariant envelope on $\Delta_{r,d}$.
\end{propos}

\begin{dem}
{ We follow the argument   in \cite[Lemma 3.2]{Vis98}.}  Observe that each $\wt{Z}_i$ maps to the closure $\ov{Z}_i$ of a stratum of $\Delta_{r,d}$. The morphism $\wt{\pi}$ is proper and $\GL_2$-invariant. To check the remaining equivariant envelope hypothesis, we can restrict to each $Z_i$. Let $V$ be a $\GL_2$-invariant subvariety of $Z_i$ and $\omega$  the generic point of $V$. The point $\omega$ is represented by an $(r+1)$-tuple $[f]:=[f_1, \dots, f_{r+1}]$ of forms of degree $d$ in $K[x,y]$ for some field extension $k \subset K$. By definition of $Z_i$, the greatest common divisor of the forms in $[f]$ has degree $i$. Therefore, there is a unique $K$-valued point $\wt{\omega}$ in $\wt{Z}_i$ mapping to $\omega$. We define $\wt{V}$ as the schematic closure of $\wt{\omega}$ in $\wt{Z}_i$. Because of the uniqueness of the generic point, $\wt{V}$ is in fact $\GL_2$-invariant.
\end{dem}

\begin{corol} \label{corol:ipi}
We have the identity:
\[
\pi_* \left( A_{\GL_2}^* ( \wt{Z}) \right) = { i_* }\left( A_{\GL_2}^* \left( \Delta_{r,d} \right) \right).
\]
\end{corol}

\begin{dem}
This is a  consequence of Proposition \ref{prop.equivariant.envelope} and Theorem \ref{thm:ipi}.
\end{dem}

For the next definition, we adopt the notation in the following diagram:
\begin{equation}\label{eq:projsone}
\xymatrix{
& \wt{Z}_i \ar[rr]^{\pi_i} \ar[dl]_{p_1} \ar[dr]^{p_2}& & \Pro(W_{d}^{\oplus r+1})\\
\Pro(W_i) & & \Pro(W_{d-i}^{\oplus r+1})
}
\end{equation}

\begin{defi}
For all $i=1, \dots, d$, {  denote by $h_{i} = p_1^\ast c_1^{eq} (\mathcal{O}_{\Pro(W_i) }(1))$  the (pull-back via the first projection of the) canonical equivariant lift} of the hyperplane class on the first  component of $\wt{Z}_i = \Pro(W_i) \times \Pro \left( W_{d-i}^{\oplus r + 1} \right)$. Also, denote by $H =  c_1^{eq} (\mathcal{O}_{ \Pro(W_d^{\oplus r+1} ) }(1))$ and by $\eta_i =  p_2^\ast c_1^{eq} (\mathcal{O}_{ \Pro \left( W_{d-i}^{\oplus r + 1} \right)}(1))$. Define the classes:
\[
\alpha_{i,k}(H):= \pi_{i_*} \left( h_i^k\right) \in A_{\GL_2}^* \left( \Pro(W_d^{\oplus r+1} ) \right).
\]
\end{defi}

\begin{propos}\label{prop.alpha}
We have the ideal identity:
\[
i_* \left( A_{\GL_2}^* \left( \Delta_{r,d} \right) \right) = \left( \alpha_{i,k}(H)\right)_{i=1, \dots, d; \; k=0, \dots, i}.
\]
\end{propos}

\begin{dem}
{ Using Corollary \ref{corol:ipi}, it suffices to show that  $\pi_* \left( A_{\GL_2}^* ( \wt{Z}) \right) $ is contained in the ideal $I$ generated by the classes $\alpha_{i,k}(H)$.  From \eqref{eq:defincl},  
\begin{equation} \label{eq:pullb}
\pi_i^\ast(H) = h_i+\eta_i.
\end{equation}

We prove that any class of the form $\pi_{i_\ast} (h_i^{k} \eta_i^{m})$ is in $I$ by induction on $m$. For $m=0$, $k>i$,  notice that the class $\pi_{i_\ast}(h_i^k)$ is a linear combination of the classes $\alpha_{i_k}$,   $k\leq i$, with coefficients in $\Z[l_1, l_2]$; hence the base case $m = 0$ is established.

Using projection formula and \eqref{eq:pullb} one obtains the following equalities.
$$
H \pi_{i_\ast}(h_i^k \eta_i^{m-1})  = \pi_{i_\ast}(h_i^k \eta_i^{m-1} \pi_i^\ast H) = \pi_{i_\ast}(h_i^k \eta_i^{m-1} (h_i+\eta_i))  = \pi_{i_\ast}(h_i^{k+1} \eta_i^{m-1} ) + \pi_{i_\ast}(h_i^{k} \eta_i^{m} ) .
$$
Solving
$$
 \pi_{i_\ast}(h_i^{k} \eta_i^{m}) = H \pi_{i_\ast}(h_i^k \eta_i^{m-1})  - \pi_{i_\ast}(h_i^{k+1} \eta_i^{m-1} )
$$
completes the inductive step.
}
\end{dem}

\begin{lemma}
\label{lemma.big.pol}
For every value of $r,d$, we have $P_{r,d}(H) \in Im(i_\ast)$.
\end{lemma}
\begin{proof}
The case $d=1$ is treated by direct inspection: it is known that  $\alpha_{1,0}, \alpha_{1,1}$ generate the ideal of relations for the integral Chow ring of Grassmannians (\cite[Theorem 5.26]{3264}). 

For $d\geq 2$, we consider the  component $\wt{Z}_2$ of the equivariant envelope and  the point $P = (0:1:0) = [xy]\in \Pro(W_2).$
The class $\pi_{2_\ast}(p_1^\ast([P]))$ is contained in the ideal of relations. The lemma is proved by showing that such class divides $P_{r,d}(H)$.

The map ${\pi_2}_{|p_1^{-1}([P])}$ has degree one onto the coordinate linear subspace of $\Pro(W_d^{\oplus r+1})$ obtained by setting to zero all homogeneous coordinates corresponding to monomials of the form $x^d$ or $y^d$. It follows that
\begin{equation}
  \pi_{2_\ast}(p_1^\ast([P])) = \left((H+dl_1)(H+dl_2)\right)^{r+1},  
\end{equation}
which indeed is a factor of $P_{r,d}(H)$.
\end{proof}

Using Proposition \ref{prop.alpha}, Lemma \ref{lemma.big.pol}  and recalling the notational convention \ref{convsubstitute}, we have obtained from (\ref{isom}) the following presentation for the Chow ring of $\Mcal_{0}(\Pro^r, d)$.

\begin{theorem}\label{thmheresthepres}
For any odd positive integer $d$, and any integer $r\geq 0$ we have
\begin{equation}\label{gen.rel}
A^* \left( \Mcal_{0}(\Pro^r, d) \right) \cong \frac{\Z[c_1,c_2]}{(\alpha_{i,k})_{i=1, \dots, d, k=0, \dots, i}}.\end{equation}
\end{theorem}
{
\begin{rmk}
While for $r=0$ the constructions in this section do not yield a space of stable maps (but rather a space of quasi maps as in \cite{CFKM}), we find convenient to include this case in our study since we will be computing some of the generating relations inductively on $r$, with $r=0$ being the base case.
\end{rmk}}


\section{Eliminating redundant relations}
\label{sec:eliminredrel}

In this section we reduce the number of generators of the ideal of relations in the quotient ring (\ref{gen.rel}). More precisely, we  prove that the ideal
\[
\left(\pi_* \left( A_{\GL_2}^* ( \wt{Z}) \right)\right)=(\alpha_{i,k})_{i=1, \dots, d, k=0, \dots, i}
\]
is generated by $\alpha_{1,0}, \alpha_{1,1},$ and $\alpha_{i,0}$ for all positive integers $i$ that are powers of a prime number.

The goal is to prove that the image of the push-forward map
\[
\pi_{i*}: A_{\GL_2}^* \left( \Pro(W_i) \times \Pro \left( W_{d-i}^{\oplus r + 1}\right) \right) \to A_{\GL_2}^* \left( \Pro \left( W_{d}^{\oplus r + 1} \right) \right)
\]

is contained in the image of the push-forward maps $\pi_{a*}$ with $a$ less than $i$, whenever $i$ is not the power of a prime number.  Consider the  commutative diagram

\begin{equation}\label{eq:projstwo}
\xymatrix{
{\displaystyle   \bigsqcup_{a=1}^{i-1}} \Pro(W_{a}) \times \Pro(W_{i-a}) \times \Pro(W_{d- i}^{\oplus r+1})  \ar[d]_{\psi_i:=\bigsqcup \pi_a \times \text{id}} \ar[rrr]^{\bigsqcup \text{id} \times \pi_{i-a}}& & & {\displaystyle   \bigsqcup_{a=1}^{i-1}} \Pro(W_{a}) \times \Pro(W_{d- a}^{\oplus r+1}) \ar[d]^{\bigsqcup \pi_a}\\
\Pro(W_i) \times \Pro(W_{d-i}^{\oplus r+1}) \ar[rrr]^{\pi_i} & & & \Pro(W_{d}^{\oplus r+1})
}
\end{equation}
in which the only morphism that has not been already defined is $\psi_i:=\bigsqcup \pi_a \times \text{id}$. In this context, the morphism $\pi_a: \Pro(W_{a}) \times \Pro(W_{i-a}) \to \Pro(W_i)$ is defined as
\[
\pi_{a} ([c(x,y)],[g(x,y)]) = [c(x,y) \cdot g(x,y)].
\]
In other words, on the left side of diagram (\ref{eq:projstwo}), the morphisms $\pi_a$ are considered in the case $r=0$, while in the rest of the diagram, the morphisms $\pi_{k}$ are considered for a fixed value of $r$.

Since the above diagram is commutative, in order to show that the image of the push-forward map $\pi_{i*}$ is contained in the image of the push-forward maps $\pi_{a*}$ with $a$ less than $i$, it is enough to show that the push-forward $\psi_{i*}$ is surjective. Moreover, it suffices to verify that $\psi_{i*}$ is surjective when considering the torus-equivariant intersection groups.
Therefore, we focus on the push-forward map:
\[
\pi_{a*}: A^{*}_{T} \left( \Pro(W_{a}) \times \Pro(W_{i-a}) \right) \to A^{*}_{T} \left( \Pro(W_i) \right)
\]
\begin{rmk}
Arguing as in Proposition \ref{prop.alpha},  the image of $\pi_{a*}$ is generated by the classes $\alpha_{a,k}(h_i):=\pi_{a*}(h_a^k)$ for $k=0, \dots, a$. Since $\pi_{a*}$ is a homomorphism of $A^*_{\GL_2}$-modules, one can  also view the image of $\pi_{a*}$ as generated by the classes
\begin{equation}\label{modified:classes}
\wt{\alpha}_{a,k}:= \pi_{a*}(P_{a,k}(h_a))
\end{equation}
where $\{ P_{a,k}(h_a) \}_{k=0, \dots, a}$ is a suitable family of monic polynomials of degree $k$ that we define as
\[
P_{a,k}(h_a):=\prod_{s=0}^{k-1}(h_a + (a-s)l_1 + s l_2).
\]
In other words, $P_{a,k}(h_a)$ is the equivariant Chow class of the locus of polynomials in $W_a$ that are divisible by $y^k$. Note that  $P_{a,0}(h_a)=1$. In the following Lemma we determine the classes $\wt{\alpha}_{a,k}$.

\end{rmk}
\begin{lemma}\label{lemma:coefficient-alpha-tilde}
Let $\wt{\alpha}_{a,k}$ be the class in $A^T_{*}(\Pro(W_i))$ defined in (\ref{modified:classes}). Then we have the following identity:
\[
\wt{\alpha}_{a,k}=\binom{i-k}{a-k} P_{i,k}(h_i) \]

\begin{dem}
Let $L_{a,k} \subset \Pro(W_a) \times \Pro(W_{i-a})$ be the locus of pairs of polynomials where the first one is a multiple of $y^k$. The locus $L_{a,k}$ is invariant for the action of the torus $T$ and its equivariant class is $P_{a,k}(h_a)$. The restriction of the map
\[
\pi_a: \Pro(W_a) \times \Pro(W_{i-a})
\]
to $L_{a,k}$ is finite of generic degree $\displaystyle \binom{i-k}{a-k}$ onto its image. This is due to the following facts:
\begin{itemize}
\item the image of $L_{a,k}$ is the locus of polynomials of degree $i$ which are divisible by $y^{k}$;
\item let $R(x,y)=y^k S(x,y)$ be a generic homogeneous polynomial of degree $i$ which is divisible by $y^k$.
Since we are in an algebraically closed field, and assuming by generality that $S(x,y)$  has $i-k$ distinct roots, there are exactly $\displaystyle \binom{i-k}{a-k}$ ways to choose a factor $F(x,y)$ of $S(x,y)$ of degree $a-k$ so that $[y^kF(x,y), S(x,y)/F(x,y)]$  in $\Pro(W_a) \times \Pro(W_{i-a})$ maps to $R(x,y)$.
\end{itemize}
To conclude the proof, we notice that $\left[L_{a,k} \right]=P_{a,k}(h_a)$ and $\left[\pi_{a}(L_{a,k}) \right]=P_{i,k}(h_i)$, therefore we have
\[
\wt{\alpha}_{a,k}=\binom{i-k}{a-k} P_{i,k}(h_i). \]
\end{dem}

\end{lemma}

\begin{theorem}\label{ideal.of.relations}
The ideal of relations $(\alpha_{i,k})_{i=1, \dots, d, k=0, \dots, i}$ is generated by the classes $\alpha_{1,0}$, $\alpha_{1,1}$ from $\wt{Z}_1$ and the pushforwards $\alpha_{i,0}$ of the fundamental classes of $\wt{Z}_i$ where $i$ is the power of a prime number.
\end{theorem}

\begin{dem}
Assume that $i>1$ and  fix a positive integer $1\leq k\leq i$. We want to show that there exists a monic polynomial in $h_i$ of degree $k$ that is in the image of $\psi_*$. For every integer $a<i$ and integer $j \leq \min(a,k)$, we have that $h_i^{k-j}\wt{\alpha}_{a,j}$ is a polynomial in $h_j$ with leading coefficient equal to $\displaystyle \binom{i-j}{a-j}$. If $k\geq 2$ we can choose $a=j=k-1$, so that the polynomial $\wt{\alpha}_{k-1,k-1}$ is monic in $h_j$. If $k=1$ then $k<i$ and $\wt{\alpha}_{1,1}$ is monic.

For the case $k=0$, the only homogeneous monic polynomial in $h_i$ of degree zero  is the fundamental class 1. We know that the greatest common divisor of the binomials $\displaystyle \binom{i}{a}$ for $a$ going from $1$ to $i-1$ is in the image of $\psi_*$. It is a consequence of Lucas's theorem \cite{Lucas} that such GCD is 1 when $i$ is not the power of a prime number (and the prime number otherwise), and this concludes the proof.
\end{dem}

\section{Relations from $Z_1$}
\label{sec:relfromz1}

In this section we compute the relations coming from the first envelope $Z_1$. We exhibit generating functions, whose coefficients are functions of the degree $d$, encoding the classes $\alpha_{1,i}^r$ for all values of $r$. 

\begin{theorem}
\label{thm:genfunctionsforfirstenveloperelations}
For $k= 0,1$, consider the $A^* \left( \Mcal_{0}(\Pro^r, d) \right)$ valued generating functions: 
$$
\mathcal{A}_{1,k}(c_1, c_2, d)  = \sum_{r=0}^{\infty} \pi_{1, \ast} [h_1^k] = \sum_{r=0}^{\infty} \alpha^{r}_{1,k},
$$
encoding the push-forwards of the fundamental class, and equivariant hyperplane class from the first component $Z_1$ of the envelope. We have:
\begin{equation}\label{eq:gf0}
  \mathcal{A}_{1,0}(c_1, c_2, d)  = \frac{d}{(1+\frac{d-1}{2}c_1)(1-\frac{d+1}{2}c_1)+d^2c_2} 
\end{equation}

\begin{equation}\label{eq:gf1}
 \mathcal{A}_{1,1}(c_1, c_2, d)  = \frac{1+\frac{d-1}{2}c_1}{(1+\frac{d-1}{2}c_1)(1-\frac{d+1}{2}c_1)+d^2c_2}-1.   
\end{equation}
\end{theorem}

We begin by describing the strategy of proof of Theorem \ref{thm:genfunctionsforfirstenveloperelations}.
The starting point  is the following commutative diagram, comparing the maps $\pi_1$ from \eqref{eq:projsone} for target dimensions $r$ and $r+1$:
\begin{equation}\label{eq:commdiagdimrr+1}
\xymatrix{
\Pro(W_1)\times \Pro(W_{d-1}^{\oplus r+1}) \ar[rr]^{\pi_1} \ar[d]_{\text{id}\times i_r } & & \Pro(W_{d}^{\oplus r+1}) \ar[d]^{i_r }\\
\Pro(W_1)\times \Pro(W_{d-1}^{\oplus r+2}) \ar[rr]^{\pi_1} & & \Pro(W_{d}^{\oplus r+2}) },
\end{equation}
where the maps denoted $i_r$ are the linear inclusions into the subspace where the  homogeneous coordinates corresponding to the $(r+2)$-th direct summand are equal to zero.
The equality of the push-forwards along the two different compositions in \eqref{eq:commdiagdimrr+1} gives rise to recursive relations that determine the classes $\alpha_{1,k}^{r+1}$ from the classes $\alpha_{1,k}^{r}$.
Appropriately organizing the recursive relations gives rise to a system of functional equations. The generating functions $\mathcal{A}_{1,k}$ from \eqref{eq:gf0} and \eqref{eq:gf1} are then shown to satify the system of functional equations and to have the correct initial conditions, corresponding to the case $r=0$ (the target is a point).

\begin{proof}
We first establish some  conventions meant to simplify notation: the two  horizontal arrows of  diagram \eqref{eq:commdiagdimrr+1} are both denoted by $\pi_1$, and we will let the context determine which map any given equation is referring to. Similarly, the equivariant hyperplane classes for the three projective spaces in both rows of \eqref{eq:commdiagdimrr+1} are denoted $h_1,\eta_1$ and $H$ (as in Section \ref{sec:genrel}); each class in the top row is the pull-back of the corresponding class in the bottom row. So, for instance, any polynomial in $H, c_1, c_2$ denotes both a class for $\Pro(W_{d}^{\oplus r+2})$ and one for $\Pro(W_{d}^{\oplus r+1})$, with the latter being the pullback via $i_r$ of the former. 

The commutativity of the diagram \eqref{eq:commdiagdimrr+1} implies that for $k=0,1$:
\begin{equation}\label{eq:pushpush}
    i_{r\ast}(\pi_{1_{\ast}}(h_1^k)) =\pi_{1_{\ast}}((\text{id}\times i_r)_\ast(h_1^k)).
\end{equation}

Starting with the left hand side of \eqref{eq:pushpush}, $\pi_{1_{\ast}}(h_1^k) = \alpha^{r}_{1,k}(H)$ by definition, and by the projection formula one has
\begin{equation}\label{eq:LHS}
  i_{r_{\ast}}(\alpha^{r}_{1,k}(H)) =\alpha^{r}_{1,k}(H) P_d(H),
  \end{equation}
where $P_d(H) := P_{0,d}(H)$ from \eqref{eq:Prd}.

For the right hand side of \eqref{eq:commdiagdimrr+1}, we  apply the projection formula to obtain
\begin{equation}
 \pi_{1_{\ast}}((\text{id}\times i_r)_\ast(h_1^k)) = h_1^k P_{d-1}(\eta_1).   
\end{equation}
Using that $H = h_1+\eta_1$ (where we omit $\pi_1^\ast$ from the notation for the pullback of the class $H$), and the presentation of the equivariant Chow ring of $\Pro(W_1)$, we may write
\begin{equation}\label{eq:remainder}
  h_1^k P_{d-1}(\eta_1) = R_{1,k,0}^{r,d}(H,c_1, c_2)+ R_{1,k,1}^{r,d}(H,c_1,c_2)h_1,   
\end{equation}
where the right hand side of \eqref{eq:remainder} is the remainder of the division of $P_{d-1}(H-h_1)$ by the polynomial
$P_{1}(h_1)=(h_1^2- c_1h_1 +c_2)$.

Pushing forward and using the projection formula, we then obtain:
\begin{equation}
 \pi_{1_{\ast}}((\text{id}\times i_r)_\ast(h_1^k)) = R_{1,k,0}^{r,d}(H,c_1, c_2)\alpha^{r+1}_{1,0}(H)+ R_{1,k,1}^{r,d}(H,c_1,c_2)\alpha^{r+1}_{1,1}(H).   
\end{equation}

From equations \eqref{eq:commdiagdimrr+1},\eqref{eq:LHS} and \eqref{eq:remainder}, by substituting $H= \frac{d+1}{2} c_1$ one obtains the square size $2$ linear system
\begin{equation}\label{eq:linsys}
 P_d \alpha_{1,k}^{r} = R_{1,k,0}^{r,d}\alpha^{r+1}_{1,0}+ R_{1,k,1}^{r,d}\alpha^{r+1}_{1,1},\end{equation}
where $k = 0,1$.

We observe that all coefficients in  \eqref{eq:linsys} are divisible by $P_{d-2}:=P_{d-2}\left(\frac{d-1}{2}c_1\right).$

\begin{claim}\label{claim:identities}

We have the following identities:
\begin{eqnarray}
\label{eq:PdQd}P_d &=& \left(
d^2 c_2-\frac{d^2 -1}{4} c_1^2 
\right)P_{d-2}\\
\label{eq:R100}R_{1,0,0}^{r,d} &=&  \left(
\frac{d+1}{2} c_1
\right) P_{d-2}\\
\label{eq:R101}R_{1,0,1}^{r,d} &=&  
-d P_{d-2}\\
\label{eq:R110}R_{1,1,0}^{r,d} &=&  dc_2P_{d-2}\\
\label{eq:R111}R_{1,1,1}^{r,d} &=&  \left(
\frac{1-d}{2}c_1
\right)P_{d-2}
\end{eqnarray}
\end{claim}

Using the results from Claim \ref{claim:identities}, proved at the end of the section, the linear system \eqref{eq:linsys} simplifies to
\begin{equation}\label{eq:linsysmatform}
\left(
d^2 c_2-\frac{d^2 -1}{4} c_1^2 
\right)\left[
\begin{array}{c}
\alpha^{r}_{1,0}\\
\alpha^{r}_{1,1}
\end{array}
\right] = 
\left[
\begin{array}{cc}

\frac{d+1}{2} c_1
 & -d\\
 d c_2 & 
\frac{1-d}{2}c_1
 
\end{array}
\right]
\left[
\begin{array}{c}
\alpha^{r+1}_{1,0}\\
\alpha^{r+1}_{1,1}
\end{array}
\right]
\end{equation}
The linear system \eqref{eq:linsysmatform} may be solved for the $\alpha_{1,k}^{r+1}$'s to obtain:
\begin{eqnarray}\label{eq:linsyssolut}
\alpha^{r+1}_{1,0}
&=&
\frac{1-d}{2}c_1\alpha^{r}_{1,0}+
d \alpha^{r}_{1,1}
\nonumber
\\
\alpha^{r+1}_{1,1}
&=&
-d c_2\alpha^{r}_{1,0}+
\frac{d+1}{2}c_1\alpha^{r}_{1,1}.
\end{eqnarray}

Together with the initial conditions
\begin{equation}\label{eq:initcondit}
 \alpha_{1,0}^{0} = d, \ \ \   \alpha_{1,1}^{0} = \frac{d+1}{2}c_1,
\end{equation}
the equations in \eqref{eq:linsyssolut} for all $r\geq 0$ determine recursively 
all values of $\alpha_{1,k}^{r}$.
To conclude the proof, it suffices to observe that \eqref{eq:linsyssolut}
and \eqref{eq:initcondit} give rise to the system of functional equations:
\begin{equation}
\label{eq:funsyssolut}    
\left\{
\begin{array}{ccc}
\mathcal{A}_{1,0}-d
&=&
\frac{1-d}{2}c_1\mathcal{A}_{1,0}+
d \mathcal{A}_{1,1}
\\
\mathcal{A}_{1,1} - \frac{d+1}{2}c_1
&=&
-d c_2\mathcal{A}_{1,0}+
\frac{d+1}{2}c_1\mathcal{A}_{1,1},
\end{array}
\right.
\end{equation}
together with the boundary conditions
\begin{equation}\label{eq:boundaryconditions}
    \mathcal{A}_{1,0}(0,0,d)  = d, \ \ \ \ \mathcal{A}_{1,1}(0,0,d)  = 0.
\end{equation}
It is then immediate to verify that \eqref{eq:funsyssolut},\eqref{eq:boundaryconditions} are satisfied by the generating functions \eqref{eq:gf0},\eqref{eq:gf1}. Thus Theorem \ref{thm:genfunctionsforfirstenveloperelations} is proved (as soon as the identities from Claim \ref{claim:identities} are established).
\end{proof}

\begin{proof}[Proof of Claim \ref{claim:identities}]
Equation \eqref{eq:PdQd} follows from the definition of the polynomial $P_d$ in terms of the weights $l_1$ and $l_2$, as defined in \eqref{eq:Prd}; substituting $H = \frac{d+1}{2}c_1$ in the polynomial $P_{0,d}(H)$ and  $H = \frac{d-1}{2}c_1$ in the polynomial $P_{0,d-2}(H)$ it is immediate to see that the factors of $P_{d-2}$ are equal to the internal factors of $P_d$. It follows that the quotient $P_d/P_{d-2}$ consists of the product of the first and last terms of $P_d$:
\begin{equation}\label{eq:quotientPdQd}
    \frac{P_d}{P_{d-2}} = \left(\frac{d+1}{2}c_1+dl_1\right)\left(\frac{d+1}{2}c_1+dl_2\right).
\end{equation}
Equation \eqref{eq:PdQd} follows from \eqref{eq:quotientPdQd} by expanding and substituting $-l_1-l_2 = c_1, l_1l_2 = c_2$.

Equations \eqref{eq:R100},\eqref{eq:R101},\eqref{eq:R110} and \eqref{eq:R111} are proved by induction on odd values of $d$,
with the base case $d=1$ easily established after noting $P_{-1}=1$.

We  work in the quotient ring $\Z[h_1, c_1,c_2]/ (h_1^2-c_1h+c_2)$, adopting the convention that for any polynomial $q$, its image $[q]$ in the quotient ring is identified with the remainder of division by $h_1^2-c_1h+c_2$.
From the definitions for the $(d+2)$-th case one may observe that
\begin{equation}\label{eq:d-1isalsonested}
    P_{d+1}\left(\frac{d+3}{2}-h_1\right) = 
    \left(\frac{d+3}{2}c_1+(d+1)l_1-h_1\right)\left(\frac{d+3}{2}c_1+(d+1)l_2- h_1\right) P_{d-1}\left(\frac{d+1}{2}-h_1\right).
\end{equation}
By \eqref{eq:remainder}
\begin{equation}
    \left[P_{d+1}\left(\frac{d+3}{2}-h_1\right)
    \right] = 
    \left[
    R^{r,d+2}_{1,0,0}+ R^{r,d+2}_{1,0,1}h_1
    \right],
\end{equation}so after symmetrizing the quadratic factor in \eqref{eq:d-1isalsonested} and replacing the last factor using the inductive hypothesis, one has
\begin{equation}
 \left[
    R^{r,d+2}_{1,0,0}+ R^{r,d+2}_{1,0,1}h_1
    \right] = 
    \left[\left(h_1^2 - 2c_1h_1 - 
    \frac{(d+1)(d-3)}{4}c_1^2
    + (d+1)^2c_2 
    \right)
    \left(
    -dh_1+\frac{d+1}{2}c_1\right) P_{d-2}
    \right]
\end{equation}
In order to obtain \eqref{eq:R100} and \eqref{eq:R101} one must check that:
\begin{equation}
 \left[
    R^{r,d+2}_{1,0,0}+ R^{r,d+2}_{1,0,1}h_1
    \right] = 
    \left[
    \left(
    -(d+2)h_1+\frac{d+3}{2}c_1\right) P_{d}
    \right] = \left[
    \left(
    -(d+2)h_1+\frac{d+3}{2}c_1\right)\left(
d^2 c_2-\frac{d^2 -1}{4} c_1^2 
\right) P_{d-2}
    \right] ,
\end{equation}
where the last equality is obtained applying \eqref{eq:PdQd}.
It is then sufficient to verify
\begin{equation}
   \left[\left(h_1^2 - 2c_1h_1 - 
    \frac{(d+1)(d-3)}{4}c_1^2
    + (d+1)^2c_2 
    \right)
    \left(
    -dh_1+\frac{d+1}{2}c_1\right) 
    \right] = \left[
    \left(
    -(d+2)h_1+\frac{d+3}{2}c_1\right)\left(
d^2 c_2-\frac{d^2 -1}{4} c_1^2 
\right) 
    \right],   
\end{equation}
which is easily done.

In order to prove \eqref{eq:R110}, \eqref{eq:R111},
one has
\begin{equation}
\left[
    R^{r,d}_{1,1,0}+ R^{r,d}_{1,1,1}h_1
    \right] =
    \left[h_1 P_{d-1}\left(\frac{d+1}{2}c_1-h_1\right)
    \right] = 
    \left[
    R^{r,d}_{1,0,0}h_1+ R^{r,d}_{1,0,1}h_1^2
    \right]=    \left[- R^{r,d}_{1,0,1}c_2+
    \left(R^{r,d}_{1,0,0}+c_1 R^{r,d}_{1,0,1}\right)h_1
    \right],
    \end{equation}
 from which the result follows by direct substitution.
\end{proof}

\section{Relations of type $\alpha_{i,0}$}
\label{sec:reloftype}


In this section we exhibit formulas for the relations obtained by pushing-forward the fundamental classes of the various components $Z_i$ of the envelope.  These classes exhibit the following interesting structure: for any fixed values of $d,r$, the relations $\alpha_{i,0}$ are obtained for all $i$ via the action of a differential operator on a single monomial function.

\begin{theorem} \label{thm:degszero} 

For any $d>0, r\geq 0, 0\leq i \leq d$, the relations $\alpha_{i,0}\in A^\ast(\Mcal_{0}(\Pro^r, d))$ are given by the formula:
\begin{equation}\label{eq:closedformzerorel}
\alpha_{i,0}=
     \displaystyle \sum_{j=0}^{i} \frac{(-1)^{j}}{i! (l_2-l_1)^i}\binom{i}{j}\left(\frac{P_{d}}
     {\prod_{k = j}^{d-i+j}\left(\frac{c_1}{2}+\left(k-\frac{d}{2}\right)(l_2-l_1) \right)}
     \right)^{r+1}.  
\end{equation}

\end{theorem}
\begin{rmk}
We observe that in order to get from formula \eqref{eq:closedformzerorel} to the generating function $\mathcal{A}_{i,0}(d)$ from the Main Theorem, it is just a matter of noticing the geometric series arising when adding  \eqref{eq:closedformzerorel} for all values of $r$.
\end{rmk}

We obtain formula \eqref{eq:closedformzerorel} through an application of Atiyah-Bott localization, which we now recall.

\begin{theorem}[{\cite[Theorem 2]{EGL}}]\label{thm.loc}
	Define the $A^*_T$-module $ \Qcal:= ((A^*_T)^+)^{-1}A^*_T$, 
	where $(A^*_T)^+$ is the multiplicative system of homogeneous elements of $A^*_T$ of positive degree.
	
	Let $X$ be a smooth $T$-variety and consider the locus $F$ of fixed points for the action of $T$. Let $F=\cup F_j$ be the decomposition of $F$ into irreducible components. For every $\gamma$ in $A^*_T(X)\otimes\Qcal$, we have the identity
	\[ \gamma = \sum_{j} \frac{i_{F_j}^*(\gamma)}{c_{top}^T(N_{F_{j}/X})} \]
	where $i_{F_j}$ is the inclusion of $F_j$ in $X$ and $N_{F_j/X}$ is the normal bundle of $F_j$ in $X$.
\end{theorem}

\begin{proof}[Proof of Theorem \ref{thm:degszero}]
Referring to diagram
\eqref{eq:projsone} for notation,  we compute 
\begin{equation}
    \alpha_{i,0} = \pi_{1_{\ast}}([\widetilde{Z}_i]) = \pi_{1_{\ast}}(p_1^\ast([\Pro(W_i)]))
\end{equation}
by first localizing the fundamental class of $[\Pro(W_i)]$ and then pull-pushing the corresponding expression taking advantage of the fact that it is supported on a torus invariant subvariety. 

For $0\leq j\leq i$, denote by $q_j$ the  torus invariant point where only the $j$-th homogeneous coordinate is non-zero. From Theorem \ref{thm.loc} we have
\begin{equation}
  [\Pro(W_i)] = \sum_{j=0}^i \frac{[q_j]}{e(N_{q_j/\Pro(W_i)})},
\end{equation}
where the equivariant Euler class of the normal bundle to $q_j$ is the product of the $i$ tangent weights 
\begin{equation}
e(N_{q_j/\Pro(W_i)}) = 
\prod_{k\not=j}[((i-k) l_1+ kl_2)-((i-j) l_1+j l_2)] = (-1)^jj!(i-j)!(l_2-l_1)^i.
\end{equation}
Since both $p_1^\ast$ and $\pi_{1_{\ast}}$ are morphisms of $\Qcal$ modules, we have
\begin{equation}\label{eq:kicaku}
  \pi_{1_{\ast}}(p_1^\ast([\Pro(W_i)])) = \sum_{j=0}^i (-1)^j\frac{\pi_{1_{\ast}}[q_j\times \Pro(W_{d-i}^{\oplus r+1})]}{j!(i-j)!(l_2-l_1)^i}  
\end{equation}

It remains to compute the equivariant Chow class of $\pi_{1_{\ast}}[q_j\times \Pro(W_{d-i}^{\oplus r+1})]$. One may see that $\pi_1$ maps the subspace $q_j\times \Pro(W_{d-i}^{\oplus r+1})$ isomorphically onto the linear subspace of $\Pro(W_{d}^{\oplus r+1})$ parameterizing  $(r+1)$-tuples of polynomials which are divisible by $x^j$ and $y^{i-j}$. The class of this subspace corresponds to the product of the linear factors of the polynomial $P_{r,d}(H)$ corresponding to the homogeneous coordinates that vanish. By equivalently dividing $P_{r,d}(H)$ by the factors corresponding to the coordinates that do not vanish, one may write:

\begin{equation} \label{eq:pushfwdallway}
  \pi_{1_{\ast}}[q_j\times \Pro(W_{d-i}^{\oplus r+1})]   
  = \frac{P_{r,d}(H)}{\prod_{k=j}^{d-i+j} (H+(d-k)l_1+kl_2)^{r+1} } 
\end{equation}
Plugging \eqref{eq:pushfwdallway} into \eqref{eq:kicaku} and substituting $H= \frac{d+1}{2} c_1$ one readily obtains the expression in \eqref{eq:closedformzerorel}.

\end{proof}

\begin{corol} \label{cor:diffop}
For any fixed values of $r,d>0$, the relations $ \alpha_{i,0}$ for all values of $i$ are encoded in the generating function:
\begin{equation}\label{eq:genfunzerorel}
 \mathcal{A}_{0}(d) =   \sum_ {i =0}^d \alpha_{i,0}  =\left[ 
\left(
\text{e}^{t_1\partial_x+t_2\partial_y}
\left(
x^{\frac{c_1}{2(l_1-l_2)}+\frac{d}{2}}
y^{\frac{c_1}{2(l_2-l_1)}+\frac{d}{2}}
\right)
\right)^{\odot r+1}_{|x=y=1}
\right]_{\left | t_1 =t_2=1, \; deg \leq dr \right.},
\end{equation}
where the symbol $\odot$ refers to a Hadamard power for a two-variable exponential power series in variables $t_1/(l_1-l_2), t_2/(l_2-l_1)$,  defined as follows:
\begin{equation}\label{hadamard}
\left(\sum_{\mu,\nu} a_{\mu,\nu}  \frac{\left(\frac{t_1}{l_1-l_2}\right)^\mu}{\mu!}\frac{\left(\frac{t_2}{l_2-l_1}\right)^\nu}{\nu!}\right)^{\odot r+1}  :=  \sum_{\mu,\nu} a_{\mu,\nu}^{r+1}  \frac{\left(\frac{t_1}{l_1-l_2}\right)^\mu}{\mu!}\frac{\left(\frac{t_2}{l_2-l_1}\right)^\nu}{\nu!}.
\end{equation}
\end{corol}
\begin{proof}

We    show that the degree $ri$ summand  of \eqref{eq:genfunzerorel} coincides with  \eqref{eq:closedformzerorel}.

The exponential differential operator is in commuting variables and hence it admits the natural Taylor expansion:
\begin{equation}\label{eq:expdiffop}
    \text{e}^{t_1\partial_x+t_2\partial_y} = \sum_{\mu, \nu}\frac{t_1^\mu}{\mu!}\frac{t_2^\nu}{\nu!}\partial_x^\mu\partial_y^\nu
\end{equation}
Applying the general summand in \eqref{eq:expdiffop} to the monomial 
$\left(
x^{\frac{c_1}{2(l_1-l_2)}+\frac{d}{2}}
y^{\frac{c_1}{2(l_2-l_1)}+\frac{d}{2}}
\right)$ and specializing $x=y=1$ one obtains:
\begin{equation} \label{eq:pakjdjhsd}
    \frac{t_1^\mu}{\mu!}\frac{t_2^\nu}{\nu!}\frac{(-1)^\mu}{(l_2-l_1)^{\mu+\nu}}
    \prod_{k=0}^{\mu-1} \left(\frac{c_1}{2}+\left(k-\frac{d}{2}\right)(l_2-l_1)\right)\prod_{m=0}^{\nu-1}\left(\frac{c_1}{2}+\left(\frac{d}{2}-m\right)(l_2-l_1)\right)
\end{equation}
Reindexing the second product in \eqref{eq:pakjdjhsd} by $m = d-k$ one may rewrite \eqref{eq:pakjdjhsd} to obtain
\begin{equation}
=\sum_{\mu,\nu}\label{eq:pakjd}
    \frac{t_1^\mu}{\mu!}\frac{t_2^\nu}{\nu!}\frac{(-1)^\mu}{(l_2-l_1)^{\mu+\nu}}
\frac{P_d}{
    \prod_{k=\mu}^{d-\nu+1} 
    \left(\frac{c_1}{2}+\left(k-\frac{d}{2}\right)(l_2-l_1)\right)}
\end{equation}
By the definition of Hadamard power in \eqref{hadamard}, taking the $(r+1)$-th power of the power series in \ref{eq:pakjd} has the effect of raising the last rational function to the power of $(r+1)$.
The Chow degree of the coefficient of $t_1^\mu t_2^\nu$ is $(r+1)(\mu+\nu) - (\mu+\nu) = r(\mu+\nu)$, from which follows that the degree $ri$ coefficient of the series is obtained by summing over all non-negative pairs of values of $\mu$ and $\nu$ adding to $i$. Setting $\mu = j, \nu = i-j$ one may immediately recognize formula \eqref{eq:closedformzerorel}.
\end{proof}

Formula \eqref{eq:closedformzerorel} allows us to also observe some structure of the relations $\alpha_{i,0}$ as the degree $d$ varies.
\begin{corol}\label{cor:pol}
For fixed values of $i$ and $r$, the coefficients of the classes $\alpha_{i,0}\in \Z[c_1,c_2]$ are polynomials in $d$ of degree $i(r+1)$.
\end{corol}

\section{Examples}
\label{sec:examples}

\subsection{The Case $d=1$.} Let us consider linear maps. The space $\Mcal_{0}(\Pro^r,1)$ is the Grassmannian of lines in $\Pro^r$. The integral Chow ring of Grassmannians is well known (see for example \cite[Theorem 5.26]{3264}), therefore, in this case, the result is not original but it allows us to show how to apply the formulas in the simplest case. Thanks to Theorem \ref{ideal.of.relations}, the integral intersection ring of $\Mcal_{0}(\Pro^r,1)$ has the form
\[
\mathbb Z [c_1,c_2]/I
\]
where $I$ is the ideal generated by the classes $\alpha_{1,0}$ and $\alpha_{1,1}$. By replacing $d=1$ in equations (\ref{eq:gf0}) and (\ref{eq:gf1}), we get that the classes $\alpha_{1,0}$ and $\alpha_{1,1}$ are the terms of total degree $r$ and $r+1$ in the power series expansion:
\[
\frac{1}{1 - c_1 + c_2} = 1 - (- c_1 + c_2) + (- c_1+c_2)^2 - \dots
\]
The generating functions $\mathcal A_{1,0}(c_1,c_2,1)$ and $\mathcal A_{1,1}(c_1,c_2,1)$ from equations (\ref{eq:gf0}) and (\ref{eq:gf1}) are the same up to the constant term. The above result is consistent with \cite[Theorem 5.26]{3264} up to replacing our $c_1$ with $-c_1$.

\subsection{The Case $r=2$ and $d=3$.} \label{sec:planecubics} We  explicitly determine the case of rational cubics in $\Pro^2$. By Theorem \ref{ideal.of.relations}, the integral intersection ring of $\Mcal_{0}(\Pro^2,3)$ has the form
$
\mathbb Z [c_1,c_2]/I,
$
where $I$ is the ideal generated by the classes $\alpha_{1,0}$, $\alpha_{1,1}$, $\alpha_{2,0}$, and $\alpha_{3,0}$.

The generating functions for $\alpha_{1,0}$ and $\alpha_{1,1}$, when $d=3$, are
\begin{eqnarray*}
\mathcal A_{1,0}(3) &=& \frac{3}{(1+c_1)(1-2c_1)+9c_2}=\frac{3}{1-c_1-2c_1^2 +9c_2}=3 \left( 1 + (c_1 + 2c_1^2 - 9c_2) + (c_1 + 2c_1^2 - 9c_2)^2 + \dots \right);\\
\mathcal A_{1,1}(3) &=& \frac{2c_1+2c_1^2 - 9c_2}{1-c_1-2c_1^2 +9c_2}=(2c_1+2c_1^2 - 9c_2) \left( 1 + (c_1 + 2c_1^2 - 9c_2) + (c_1 + 2c_1^2 - 9c_2)^2 + \dots \right). \\
\end{eqnarray*}

When $r=2$, the class $\alpha_{1,0}$ is the term of degree 2 of $\mathcal A_{1,0}(c_1,c_2,3)$. More precisely:
\[
\alpha_{1,0}=3(2c_1^2 - 9c_2 + c_1^2) = 9c_1^2 - 27 c_2.
\]
On the other hand, the class $\alpha_{1,1}$ is the term of degree 3 of $\mathcal A_{1,1}(c_1,c_2,3)$, that is to say:
\[
\alpha_{1,1}=2c_1(3c_1^2 - 9c_2)+(2c_1^2 - 9c_2)(c_1) =8 c_1^3 - 27 c_1 c_2.
\]

We apply formula (\ref{eq:closedformzerorel}) to determine $\alpha_{2,0}$ and $\alpha_{3,0}$. We  use the following script on Maple that  works for any values of $d$ and $r$:
\begin{center}
\includegraphics[scale=0.4]{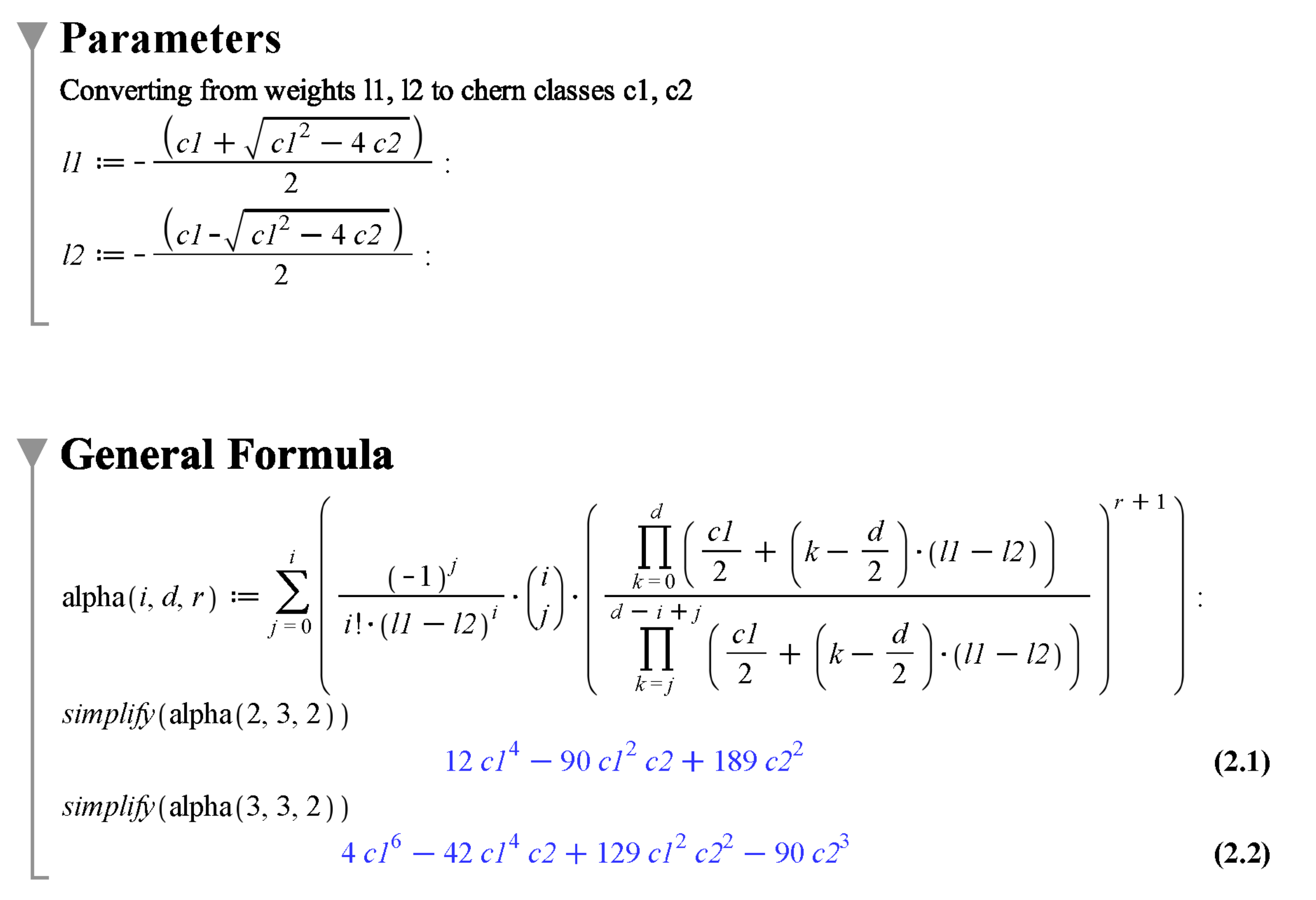}
\end{center}

Hence we have:
\begin{eqnarray*}
\alpha_{2,0}&=&12c_1^4 - 90c_1^2c_2 + 189c_2^2;\\
\alpha_{3,0}&=& 4c_1^6 - 42c_1^4c_2 + 129c_1^2c_2^2 - 90c_2^3.
\end{eqnarray*}

One can simplify the generators to write the final result in a more compact form:

\begin{tabular}{l|l}
Classes & Reduction\\
\hline
& \\
$\alpha_{1,0}=9 c_1^2 - 27c_2$ & \fbox{$9 c_1^2 - 27c_2$}\\
& \\
$\alpha_{1,1}=8c_1^3 - 27c_1 c_2$ & $\text{\fbox{$c_1^3$}}=c_1\cdot \alpha_{1,0}-\alpha_{1,1}$\\
& moreover, $27c_1 c_2$ is generated by $c_1^3$ and $\alpha_{1,1}$\\
& \\
$\alpha_{2,0}=12c_1^4 - 90c_1^2 c_2 + 189c_2^2$ & $\alpha_{2,0} = (4c_1^2 - 7c_2)\cdot \alpha_{1,0} - 3c_1 \cdot \alpha_{1,1}$ \\ 
& \\
$\alpha_{3,0}=4c_1^6 - 42c_1^4c_2+129c_1^2c_2^2 - 90c_2^3$ &

{$\fbox{$6c_2^2c_1^2+9c_2^3$}= (4c1^4 - 10c1^2c2 + 3c2^2)\alpha_{1,0} - (4c1^3 - 6c1c2 )\alpha_{1,1} -\alpha_{3,0}$}
\end{tabular}

\bigskip

The class $\alpha_{2,0}$ is redundant since it belongs to the ideal generated by $\alpha_{1,0}$ and $\alpha_{1,1}$.
 On the other hand, the class $\alpha_{3,0}$ does not belong to the ideal generated by $\alpha_{1,0}$ and $\alpha_{1,1}$. One may easily check this as follows: taking a further quotient by the ideal generated by $c_1$ and $27$, we get $\alpha_{1,0}\equiv \alpha_{1,1}\equiv 0$ but $\alpha_{3,0}\equiv 9c_2^3 \not\equiv 0$.
One can also see that 
$$3(6c_2^2c_1^2+9c_2^3) = -c_2^2(9c_1^2 -27c_2) +c_1c_2(27c_1c_2),$$
showing in particular that the class $3\alpha_{3,0}$ is in the ideal generated by $\alpha_{1,0}$ and $\alpha_{1,1}$. In conclusion, we have shown the following result.
 
\begin{theorem}\label{thm.r2.d3}
We have the isomorphism:
\[
A^* \left( \Mcal_{0}\left( \Pro^2, 3 \right) \right) \cong \frac{\Z[c_1,c_2]}{(9c_1^2 - 27c_2,\; c_1^3,\; 6c_2^2c_1^2+9c_2^3)}.
\]
\end{theorem}

As seen in Theorem \ref{thm.r2.d3}, the set of generators described in Theorem \ref{ideal.of.relations} is not necessarily minimal. We propose the following conjecture.

\begin{conj}\label{conj:gens}
Let $r$ be a positive integer and $d$ a positive odd number. Then
\[
A^* \left( \Mcal_{0}\left( \Pro^r, d \right) \right) \cong \frac{\Z[c_1,c_2]}{(\alpha_{1,0}, \alpha_{1,1}, \{ \alpha_{p,0} | \; \text{$p$ is a prime that divides $d$} \}) }.
\]
Further, all the relations listed are necessary.
\end{conj}

Seth Ireland  programmed a {\tt Macaulay2} code that provided extensive verification for this conjecture \cite{Seth}.  At this point we know the conjecture to be true for $r \leq 9$, $d\leq 49 $. A weaker version of the conjecture, asserting generation but not minimality, has been verified for $r\leq 5 $, $d\leq 99$.


\bibliographystyle{alpha} 
\bibliography{Kontsevich} 

\newcommand{\etalchar}[1]{$^{#1}$}
\begin{thebibliography}{HKK{\etalchar{+}}03}

\bibitem[BM96]{BM}
K.~Behrend and Yu. Manin.
\newblock Stacks of stable maps and {G}romov-{W}itten invariants.
\newblock {\em Duke Math. J.}, 85(1):1--60, 1996.

\bibitem[CFKM14]{CFKM}
Ionu{\c{t}} Ciocan-Fontanine, Bumsig Kim, and Davesh Maulik.
\newblock {Stable quasimaps to GIT quotients}.
\newblock {\em J. Geom. Phys.}, 75:17--47, 2014.

\bibitem[CL21]{CanLar}
Samir Canning and Hannah Larson.
\newblock { The Chow rings of the moduli spaces of curves of genus 7, 8, and
  9}.
\newblock {Preprint} arXiv:2104.05820, 2021.

\bibitem[DL21]{DiLh}
A.~Di~Lorenzo.
\newblock Cohomological invariants of the stack of hyperelliptic curves of odd
  genus.
\newblock {\em Transform. Groups}, 26(1):165--214, 2021.

\bibitem[DLFV21]{DLFV}
Andrea Di~Lorenzo, Damiano Fulghesu, and Angelo Vistoli.
\newblock The integral {C}how ring of the stack of smooth non-hyperelliptic
  curves of genus three.
\newblock {\em Trans. Amer. Math. Soc.}, 374(8):5583--5622, 2021.

\bibitem[dLPV21]{DilPVis}
Andrea di~Lorenzo, Michele Pernice, and Angelo Vistoli.
\newblock { Stable cuspidal curves and the integral Chow ring of
  $\overline{\mathcal{M}}_{2,1}$}.
\newblock {Preprint} arXiv:2108.03680, 2021.

\bibitem[Edi13]{Ed:eg}
Dan Edidin.
\newblock Equivariant geometry and the cohomology of the moduli space of
  curves.
\newblock In {\em Handbook of moduli. {V}ol. {I}}, volume~24 of {\em Adv. Lect.
  Math. (ALM)}, pages 259--292. Int. Press, Somerville, MA, 2013.

\bibitem[EF09]{EFh}
Dan Edidin and Damiano Fulghesu.
\newblock The integral {C}how ring of the stack of hyperelliptic curves of even
  genus.
\newblock {\em Math. Res. Lett.}, 16(1):27--40, 2009.

\bibitem[EG98a]{EG}
Dan Edidin and William Graham.
\newblock Equivariant intersection theory.
\newblock {\em Invent. Math.}, 131(3):595--634, 1998.

\bibitem[EG98b]{EGL}
Dan Edidin and William Graham.
\newblock Localization in equivariant intersection theory and the {B}ott
  residue formula.
\newblock {\em Amer. J. Math.}, 120(3):619--636, 1998.

\bibitem[EH16]{3264}
David Eisenbud and Joe Harris.
\newblock {\em 3264 and all that---a second course in algebraic geometry}.
\newblock Cambridge University Press, Cambridge, 2016.

\bibitem[Fab90a]{Fab:Ch1}
Carel Faber.
\newblock Chow rings of moduli spaces of curves. {I}. {T}he {C}how ring of
  {$\overline{M}_3$}.
\newblock {\em Ann. of Math. (2)}, 132(2):331--419, 1990.

\bibitem[Fab90b]{Fab:Ch2}
Carel Faber.
\newblock Chow rings of moduli spaces of curves. {II}. {S}ome results on the
  {C}how ring of {$\overline{ M}_4$}.
\newblock {\em Ann. of Math. (2)}, 132(3):421--449, 1990.

\bibitem[FP97]{FP}
W.~Fulton and R.~Pandharipande.
\newblock Notes on stable maps and quantum cohomology.
\newblock In {\em Algebraic geometry---{S}anta {C}ruz 1995}, volume~62 of {\em
  Proc. Sympos. Pure Math.}, pages 45--96. Amer. Math. Soc., Providence, RI,
  1997.

\bibitem[Ful98]{Ful}
William Fulton.
\newblock {\em Intersection theory}, volume~2 of {\em Ergebnisse der Mathematik
  und ihrer Grenzgebiete. 3. Folge. A Series of Modern Surveys in Mathematics
  [Results in Mathematics and Related Areas. 3rd Series. A Series of Modern
  Surveys in Mathematics]}.
\newblock Springer-Verlag, Berlin, second edition, 1998.

\bibitem[FV18]{Fu-Vi}
Damiano Fulghesu and Angelo Vistoli.
\newblock The {C}how ring of the stack of smooth plane cubics.
\newblock {\em Michigan Math. J.}, 67(1):3--29, 2018.

\bibitem[Hei05]{Hei05}
J.~Heinloth.
\newblock Notes on differentiable stacks.
\newblock In {\em Mathematisches {I}nstitut, {G}eorg-{A}ugust-{U}niversit\"{a}t
  {G}\"{o}ttingen: {S}eminars {W}inter {T}erm 2004/2005}, pages 1--32.
  Universit\"{a}tsdrucke G\"{o}ttingen, G\"{o}ttingen, 2005.

\bibitem[Hil30]{Hilbert}
David Hilbert.
\newblock Probleme der {G}rundlegung der {M}athematik.
\newblock {\em Math. Ann.}, 102(1):1--9, 1930.

\bibitem[HKK{\etalchar{+}}03]{Mirror-symmetry-book}
Kentaro Hori, Sheldon Katz, Albrecht Klemm, Rahul Pandharipande, Richard
  Thomas, Cumrun Vafa, Ravi Vakil, and Eric Zaslow.
\newblock {\em Mirror symmetry}, volume~1 of {\em Clay Mathematics Monographs}.
\newblock American Mathematical Society, Providence, RI, 2003.
\newblock With a preface by Vafa.

\bibitem[Ire22]{Seth}
Seth Ireland.
\newblock Personal communication. {C}ode available upon request, {\tt
  seth.ireland@colostate.edu}.
\newblock 2022.

\bibitem[Iza95]{Iza}
E.~Izadi.
\newblock The {C}how ring of the moduli space of curves of genus {$5$}.
\newblock In {\em The moduli space of curves ({T}exel {I}sland, 1994)}, volume
  129 of {\em Progr. Math.}, pages 267--304. Birkh\"{a}user Boston, Boston, MA,
  1995.

\bibitem[Kon95]{Kon:WC}
Maxim Kontsevich.
\newblock Enumeration of rational curves via torus actions.
\newblock In {\em The moduli space of curves ({T}exel {I}sland, 1994)}, volume
  129 of {\em Progr. Math.}, pages 335--368. Birkh\"{a}user Boston, Boston, MA,
  1995.

\bibitem[Lar21]{LarM2bar}
Eric Larson.
\newblock The integral {C}how ring of {$\overline M_2$}.
\newblock {\em Algebr. Geom.}, 8(3):286--318, 2021.

\bibitem[Luc78]{Lucas}
Edouard Lucas.
\newblock Theorie des {F}onctions {N}umeriques {S}implement {P}eriodiques.
\newblock {\em Amer. J. Math.}, 1(4):289--321, 1878.

\bibitem[MRV06]{MV}
Luis~Alberto Molina~Rojas and Angelo Vistoli.
\newblock On the {C}how rings of classifying spaces for classical groups.
\newblock {\em Rend. Sem. Mat. Univ. Padova}, 116:271--298, 2006.

\bibitem[Mum83]{Mum}
David Mumford.
\newblock Towards an enumerative geometry of the moduli space of curves.
\newblock In {\em Arithmetic and geometry, {V}ol. {II}}, volume~36 of {\em
  Progr. Math.}, pages 271--328. Birkh\"{a}user Boston, Boston, MA, 1983.

\bibitem[Pan98]{Pandhanonlin}
Rahul Pandharipande.
\newblock The {C}how ring of the nonlinear {G}rassmannian.
\newblock {\em J. Algebraic Geom.}, 7(1):123--140, 1998.

\bibitem[PV15]{PV}
Nikola Penev and Ravi Vakil.
\newblock The {C}how ring of the moduli space of curves of genus six.
\newblock {\em Algebr. Geom.}, 2(1):123--136, 2015.

\bibitem[Tot99]{Tot:CR}
Burt Totaro.
\newblock The {C}how ring of a classifying space.
\newblock In {\em Algebraic {$K$}-theory ({S}eattle, {WA}, 1997)}, volume~67 of
  {\em Proc. Sympos. Pure Math.}, pages 249--281. Amer. Math. Soc., Providence,
  RI, 1999.

\bibitem[Vis98]{Vis98}
Angelo Vistoli.
\newblock The {C}how ring of {$ M_2$}. {A}ppendix to ``{E}quivariant
  intersection theory'' [{I}nvent. {M}ath. {\bf 131} (1998), no. 3, 595--634;
  {MR}1614555 (99j:14003a)] by {D}. {E}didin and {W}. {G}raham.
\newblock {\em Invent. Math.}, 131(3):635--644, 1998.

\end{thebibliography}
\end{document}